\documentclass[abstracton]{scrartcl}

\usepackage[headsepline]{scrpage2}
\usepackage[latin1]{inputenc}
\usepackage{amsfonts}
\usepackage{amsmath}
\usepackage{amssymb}
\usepackage{amsthm}

\usepackage{url}

\usepackage[pdftex,colorlinks,breaklinks,linkcolor=black,citecolor=black,filecolor=black,menucolor=black,urlcolor=black,pdfauthor={Andreas Weber},pdftitle={Tensor products of recurrent hypercyclic semigroups},plainpages=false,pdfpagelabels,bookmarksnumbered=true]{hyperref}

\ihead{A. Weber: Tensor products of recurrent hypercyclic semigroups}

\newtheorem{theorem}{Theorem}[section]

\newtheorem{proposition}[theorem]{Proposition}
\newtheorem{corollary}[theorem]{Corollary}

\theoremstyle{definition}
\newtheorem{definition}[theorem]{Definition}

\theoremstyle{remark}

\newcommand{\N}{\mathbb{N}}

\newcommand*\spa{\mathrm{span}}

\title{Tensor products of recurrent hypercyclic semigroups}
\author{Andreas Weber\footnote{Email: andreas.weber@math.uni-karlsruhe.de,
						   Address:  Englerstr. 2, 76128 Karlsruhe, Germany.}\\
						   {\large  Institut f\"ur Algebra und Geometrie,
						   Universit\"at Karlsruhe (TH)} }						   

\date{}

\begin{document}

\maketitle 
\begin{abstract} 
We study tensor products of strongly continuous semigroups on Banach spaces that satisfy the hypercyclicity criterion, the recurrent hypercyclicity criterion or are chaotic.  \\
\noindent{\em Keywords:} Hypercyclic semigroups, recurrent hypercyclic semigroups, 
	tensor products of strongly continuous semigroups.
%{\em 2000 MSC numbers:} 
\end{abstract}		

\section{Introduction and preliminaries}

In this note we study tensor products $T(t)\otimes S(t)$ of strongly continuous semigroups
$T(t)$ and $S(t)$ acting on Banach spaces $X$ and $Y$. If $\alpha$ denotes a uniform 
crossnorm on the (algebraic) tensor product $X\otimes Y$ we denote by $X\tilde{\otimes}_{\alpha}Y$
the completion of the normed space $(X\otimes Y,\alpha)$. Our main purpose is to show that for 
strongly continuous semigroups $T(t), S(t)$ satisfying the recurrent hypercyclicity criterion 
and a uniform crossnorm $\alpha$ on  $X\otimes Y$ the semigroup $T(t)\otimes S(t)$ 
acting on $X\tilde{\otimes}_{\alpha}Y$
satisfies the recurrent hypercyclicity criterion, too.  
An important ingredient in the proof of this result is the work by
Desch and Schappacher in \cite{MR2184060}. 
Our result is of particular interest when one
is working with $L^p$ spaces of the form $L^p(M_1\times M_2, \mu_1\otimes\mu_2), p\geq 1$, 
for measure spaces $(M_i,\mu_i), i=1,2$, as there is a uniform crossnorm $\alpha$ such that
$L^p(M_1\times M_2, \mu_1\otimes\mu_2) = L^p(M_1,\mu_1)\tilde{\otimes}_{\alpha} L^p(M_2,\mu_2)$, cf. \cite{MR1209438}. 
Applications of our results to $L^p$ heat semigroups on certain Riemannian manifolds 
are contained in \cite{Ji:nr}.

Similar results for tensor products of semigroups or operators  can be found in 
\cite{MR2184057,MR2010778}.
\subsection{Hypercyclic and recurrent hypercyclic semigroups}

A strongly continuous semigroup $T(t)$ on a Banach space $X$ is called {\em hypercyclic}
if there exists an $x\in X$ such that its orbit $\{ T(t)x : t\geq 0\}$ is dense in $X$.\\
If additionally the set of periodic points 
$\{ x\in X : \exists t>0 \mbox{~such that~} T(t)x=x\}$ is dense in $X$, the semigroup $T(t)$
is called {\em chaotic}.\\

It is well known that a strongly continuous semigroup $T(t)$ on a separable Banach space 
$X$ is hypercyclic if and only if it is {\em topological transitive}, i.e. for any pair of
non-empty open subsets $U,V \subset X$ there exists some $t>0$ with
$T(t)U\cap V \neq\emptyset$, cf. \cite{MR1468101}.

A sufficient condition for hypercyclicity is given by the so-called {\em hypercyclicity criterion},
cf. \cite{MR2152389} for this variant:
\begin{definition}\textup{(Hypercyclicity Criterion).}\label{hypercyclicity criterion}
 A strongly continuous semigroup $T(t)$ on a separable Banach space $X$ satisfies 
 the hypercyclicity criterion if for all non-empty open subsets
 $U,V,W\subset X$ with $0\in W$ there exists a $t>0$ such that
 $$
   T(t)U \cap W \neq\emptyset \qquad\mbox{and}\qquad T(t)W\cap V \neq\emptyset.
 $$
 (Note that the same $t$ is used in both cases.)
\end{definition}
It should be remarked that a strongly continuous semigroup $T(t)$ on $X$ satisfies the hypercyclicity criterion if and only if the semigroup 
$$ T(t) \times T(t) :=
  \begin{pmatrix}
    T(t) & 0\\
    0    & T(t)
  \end{pmatrix}
$$  
is hypercyclic on $X\times X$, cf. \cite[Theorem 2.5]{MR2152389}. 
This result easily 
generalizes (with the same proof) to:
\begin{proposition}\label{nproducts}
Let $T(t)$ denote a strongly continuous semigroup on a Banach Space $X$ that satisfies
the hypercyclicity criterion. Then the diagonal semigroup 
$
T^n(t):=   T(t)\times\cdots\times T(t)
$
is hypercylic on $X^n:= X\times\cdots\times X$ for any natural number $n\geq 1$.
\end{proposition}
The discrete version of the next corollary can be found in \cite[Corollary 6]{MR2025025}.
\begin{corollary}\label{corollary nproducts}
Let $T(t)$ denote a strongly continuous semigroup on a Banach Space $X$ that satisfies
the hypercyclicity criterion. Then the diagonal semigroup $T^n(t)$ satisfies the
hypercyclicity criterion on $X^n$, $n\geq 1$, too.
\end{corollary}
\begin{proof}
 It follows from Proposition \ref{nproducts} that the semigroup $T^n(t)\times T^n(t)= T^{2n}(t)$
 is hypercyclic, and hence, by \cite[Theorem 2.5]{MR2152389}, the semigroup $T^n(t)$
 satisfies the hypercyclicity criterion. 
\end{proof}
As in \cite{MR2184060} we say that a strongly continuous semigroup satisfies the  recurrent hypercyclicity criterion if for open subsets as in Definition \ref{hypercyclicity criterion}
the set of all times $t>0$ with 
$T(t)U \cap W \neq\emptyset$ and $T(t)W\cap V \neq\emptyset$
does not have arbitrarily large holes: 
\begin{definition}
 A strongly continuous semigroup $T(t)$ on a separable Banach space $X$ satisfies
 the {\em recurrent hypercyclicity criterion} if for all non-empty open subsets $U,V,W\subset X$ with
 $0\in W$ there exists a constant $L\geq 0$ such that each interval $[t,t+L]$ containes
 an $s$ with
 $$
  T(s)U \cap W \neq \emptyset\qquad\mbox{and}\qquad T(s)W\cap V\neq\emptyset.
 $$ 
\end{definition}
Of course, any  semigroup that satisfies the recurrent hypercyclicity criterion 
is hypercyclic as it satisfies the hypercyclicity criterion.

%%%%%%%%%%%%%%
\subsection{Tensor products}

For Banach spaces $X$ and $Y$ we denote by $X\otimes Y$ their (algebraic) tensor product. 
Furthermore, let $\alpha$ be a {\em tensor norm} (or {\em uniform crossnorm})
on $X\otimes Y$ (for a definition see \cite[12.1]{MR1209438} or \cite[6.1]{MR1888309}).
Then $\alpha$ is in particular a reasonable crossnorm on $X\otimes Y$ which implies that
for $x\in X$ and $y\in Y$ we have
$$
 \alpha(x\otimes y) = ||x||_X\cdot ||y||_Y.
$$
If we define  for $z\in X\otimes Y$ 
$$
 \pi(z) = \inf \left\{ \sum_{i=1}^n ||x_i||_X\cdot ||y_i||_Y : z=\sum_{i=i}^n x_i\otimes y_i \right\}
$$
this yields a tensor norm and is called {\em projective norm}. Actually, this norm is the 
greatest reasonable crossnorm on $X\otimes Y$, i.e. if $\alpha$ is another reasonable 
crossnorm it follows
$\alpha\leq \pi$ (cf.  \cite[p. 64]{MR0094675} or \cite[Proposition 6.1]{MR1888309}).
For any norm $\alpha$ on $X\otimes Y$ we denote by $X\tilde{\otimes}_{\alpha}Y$
the completion of the normed space $(X\otimes Y,\alpha)$.

For bounded operators $T: X\to X$, $S: Y\to Y$, and any uniform crossnorm 
$\alpha$  the tensor product $T\otimes S$ is a bounded operator on $(X\otimes Y,\alpha)$ by
definition of a uniform crossnorm.  The unique extension of $T\otimes S$ to 
$X\tilde{\otimes}_{\alpha} Y$ is, for simplicity, also denoted by $T\otimes S$.

Similarly, if $T(t): X\to X$ and $S(t): Y\to Y$ are strongly continuous semigroups their tensor product
$T(t)\otimes S(t)$ is a strongly continuous semigroup on $(X\otimes Y,\alpha)$ for 
any uniform crossnorm $\alpha$. To see this, let $z\in X\otimes Y$. Then we have
\begin{eqnarray*}
 \alpha(T(t)\otimes S(t) z - z)  &\leq& \alpha(T(t)\otimes S(t) z - T(t)\otimes I z) +
 							 \alpha(T(t)\otimes I z - z)\\
					     &\leq& \pi(T(t)\otimes S(t) z - T(t)\otimes I z) +
 													    \pi(T(t)\otimes I z - z).
\end{eqnarray*}
For the first term on the right hand side it follows if $z=\sum_i x_i\otimes y_i$
is any representation of $z$
$$
 \pi(T(t)\otimes S(t) z - T(t)\otimes I z) \leq \sum_i ||T(t)x_i ||_X\cdot ||S(t)y_i-y_i||_Y \to 0 
 	\quad (t\to 0^+).
$$
As an analogous argument shows that the second term goes to zero if $t\to 0^+$, it follows
that $T(t)\otimes S(t)$ is strongly continuous.
%

%Furthermore, let $\alpha$ be a uniform cross norm
%on $X\otimes Y$ (for a definition see \cite{MR1209438}). Then we denote by 
%$X\tilde{\otimes}_{\alpha}Y$ the completion of the normed space $(X\otimes Y,\alpha)$.
%For bounded operators $T$ and $S$ on $X$ and $Y$ their tensor product
%$T\otimes S$ is a bounded operator on  $X\tilde{\otimes}_{\alpha}Y$ with
%$||T\otimes S|| = ||T||\cdot ||S||$.\\
%If we consider analytic semigroups $T(t)=e^{tA}$ and $S(t)=e^{tB}$ on $X$ and $Y$ we have the following.

%%%
\section{Main results}
\begin{theorem}\label{main thm}
Let $T(t), S(t)$ denote  strongly continuous  semigroups on Banach spaces $X$ and $Y$, respectively,
 and assume that $T(t)$ satisfies the recurrent hypercyclicity criterion. 
 Furthermore, $\alpha$ denotes a uniform crossnorm on $X\otimes Y$. 
\begin{itemize}
 \item[\textup{(a)}] If  $S(t)$  satisfies the hypercyclicity criterion, 
 	the semigroup $T(t)\otimes S(t)$ 
 		on $X\tilde{\otimes}_{\alpha} Y$ satisfies the hypercyclicity criterion, too.
 \item[\textup{(b)}] If $S(t)$ satisfies the recurrent hypercyclicity criterion, 
 	the semigroup $T(t)\otimes S(t)$ on $X\tilde{\otimes}_{\alpha} Y$ satisfies the 
	recurrent hypercyclicity criterion, too.
\end{itemize}
\end{theorem}
\begin{proof}
In the following, we use $||(x,y)|| = \sup\{||x||_X, ||y||_Y\}$ as norm on the product $X\times Y$.
Note, that the topology induced by this norm coincides with the usual product topology.
As $\alpha$ is a reasonable crossnorm, the canonical bilinear map
$$
 \psi: (X\times Y, ||\cdot||) \to (X\otimes Y,\alpha),\, (x,y)\mapsto x\otimes y
$$
is continuous and has norm $\leq 1$ (cf. \cite[p. 64]{MR0094675}). Hence, for any $n\geq 1$,
the mapping
$$
 \psi_n: \left\{
 		\begin{array}{lcl}
		 X^n\times Y^n & \to & X\otimes Y\\
		 (x_1,\ldots,x_n,y_1,\ldots, y_n) &\mapsto & \sum_{k=1}^n \psi(x_k,y_k)
		\end{array}
	\right.	
$$
is continuous for the norm 
$|| (x_1,\ldots,x_n,y_1,\ldots, y_n)|| =  \sup\{||x_k||_X, ||y_k||_Y : k=1,\ldots,n\}$ on
$ X^n\times Y^n$. 

For the proof of part (a) we proceed as follows:
Let $U, V, W$ be non-empty open subsets of $X\tilde{\otimes}_{\alpha} Y$ with $0\in W$. 
As $X\otimes Y=\spa\left(\psi(X\times Y)\right)$ is dense in 
$X\tilde{\otimes}_{\alpha} Y$, we find elements
$$
 \sum_{k=1}^m x_k\otimes y_k \in U
$$
and
$$
 \sum_{k=1}^n p_k\otimes q_k \in V.
$$
Extending one of the sums by zero summands if necessary we may assume $m=n$.
Then $\psi_n^{-1}(U)$ and $\psi_n^{-1}(V)$ are non-empty open subsets of
$X^n\times Y^n$.  Furthermore, $0 \in \psi_n^{-1}(W)$.

As $T(t)$ satisfies the recurrent hypercyclicity criterion and $S(t)$ satisfies the hypercyclicity criterion, it follows from
 \cite[Theorem 5.1]{MR2184060} that the semigroup
 $$
  \begin{pmatrix}
    T(t) & 0\\
    0    & S(t)
  \end{pmatrix} : X\times Y \to X\times Y
 $$
satisfies the hypercyclicity criterion and hence,
by Corollary \ref{corollary nproducts}, the semigroup $T^n(t)\times S^n(t)$ satisfies the 
hypercyclicity criterion, too.
Therefore, there exists $t>0$ such that 
$$
\Big(T^n(t)\times S^n(t)\psi_n^{-1}(U)\Big)\cap \psi_n^{-1}(W) \neq \emptyset
$$
and
$$
\Big(T^n(t)\times S^n(t)\psi_n^{-1}(W)\Big)\cap \psi_n^{-1}(V) \neq \emptyset.
$$
As 
$\psi_n(T^n(t)\times S^n(t)\psi_n^{-1}(U)) \subset T(t)\otimes S(t) U$ and
$\psi_n(\psi_n^{-1}(V)) \subset V$ the proof of part (a) is complete. 

For the proof of part (b) let   
$U,V,W \subset X\tilde{\otimes}_{\alpha} Y$ be non-empty open subsets  with $0\in W$. 
As in part (a) we find an $n\in \N$ such that the sets 
$\psi_n^{-1}(U), \psi_n^{-1}(V),$ and $\psi_n^{-1}(W)$ are non-empty open subsets of $X^n\times Y^n$
with $0\in \psi_n^{-1}(W)$.
Since both semigroups $T(t)$ and $S(t)$ satisfy the recurrent hypercyclicity criterion, it follows from
\cite[Corollary 5.6]{MR2184060} that the semigroup 
$T^n(t)\times S^n(t)$ satisfies the recurrent hypercyclicity criterion, too. One can now conclude as
in the proof of part (a) that the semigroup $T(t)\otimes S(t)$ satisfies the recurrent hypercyclicity criterion.
%%%%%%
%We first prove part (a).
%
%Let now $U,V,W \subset X\otimes Y$ be non-empty open subsets of the normed space
%$(X\otimes Y,\alpha)$ with $0\in W$. 
%Then $\psi^{-1}(U), \psi^{-1}(V), \psi^{-1}(W)$ are non-empty open subsets of $X\times Y$
%with $0\in \psi^{-1}(W)$. As  
% $$
%  \begin{pmatrix}
%    T(t) & 0\\
%    0    & S(t)
%  \end{pmatrix} : X\times Y \to X\times Y
% $$
%satisfies the hypercyclicity criterion, there exists some $t\geq 0$ such that
% $$
%  \begin{pmatrix}
%    T(t) & 0\\
%    0    & S(t)
%  \end{pmatrix}\psi^{-1}(U) \cap \psi^{-1}(W) \neq \emptyset
% $$
% and
%  $$
%  \begin{pmatrix}
%    T(t) & 0\\
%    0    & S(t)
%  \end{pmatrix}\psi^{-1}(W) \cap \psi^{-1}(V) \neq \emptyset.
% $$
% Hence, it follows for their image with respect to the map $\psi$
% $$
%   \Big(T(t)\otimes S(t)U\Big) \cap W \neq\emptyset  
% $$
% and
% $$
% \Big(T(t)\otimes S(t) W\Big)\cap V \neq\emptyset.
% $$
% Since $X\otimes Y$ is a dense open subset of $X\tilde{\otimes}_{\alpha}Y$ the claim follows.\\
% 
% Part (b) can be proved in an analogous manner if we use the fact that
%  $$
%  \begin{pmatrix}
%    T(t) & 0\\
%    0    & S(t)
%  \end{pmatrix} : X\times Y \to X\times Y
% $$
% satisfies the recurrent hypercyclicity criterion if both semigroups $T(t)$ and $S(t)$ satisfy the recurrent hypercyclicity criterion, cf. \cite[Corollary 5.6]{MR2184060}.
\end{proof}

\begin{corollary}
Let $T(t), S(t)$ denote  chaotic  semigroups on Banach spaces $X$ and $Y$.
If $\alpha$ denotes a uniform cross norm on the algebraic tensor product $X\otimes Y$ the semigroup $T(t)\otimes S(t)$ on $X\tilde{\otimes}_{\alpha} Y$ satisfies the recurrent hypercyclicity criterion.
\end{corollary}
\begin{proof}
 This follows directly from Theorem \ref{main thm} since any chaotic semigroup satisfies the recurrent
 hypercyclicity criterion, cf. \cite[Corollary 6.2]{MR2184060}.
\end{proof}

In order to state the next corollary, we need some preparation. Let $T$ denote a bounded operator
on a Banach space $X$. $T$ is called chaotic, if -- similar to the case of semigroups --
there is an $x\in X$ whose orbit $\{ T^n x : n\in \N \}$ is dense in $X$ and if the set of periodic
points $\{ x\in X : \exists n\in\N\mbox{~such that~} T^n x = x\}$ is dense in $X$ as well.
\begin{corollary}
 Let $T(t), S(t)$ denote strongly continuous semigroups on Banach spaces $X$ and $Y$ and
 $\alpha$ a uniform crossnorm.
 \begin{itemize}
  \item[\textup{(a)}] If there is a $t_0>0$ such that $T(t_0)$ is a chaotic operator  and $S(t_0)$ has 
  a dense set
  of periodic points, the semigroup $T(t)\otimes S(t)$ is chaotic.
  \item[\textup{(b)}] If there are $p_1, p_2, q_1, q_2\in\N$ such that $T(p_1/q_1)$ 
  	and $S(p_2/q_2)$ are chaotic the tensor product $T(t)\otimes S(t)$ is chaotic.
 \end{itemize}
\end{corollary}
\begin{proof}
We first prove (a).
From \cite[Corollary 1.12]{MR2010778} it follows that the operator $T(t_0)\otimes S(t_0)$
is chaotic and hence, the semigroup $T(t)\otimes S(t)$ is chaotic.\\
To show (b) we first remark that both semigroups $T(t)$ and $S(t)$ are chaotic and hence their
tensor product satisfies the recurrent hypercyclicity criterion. It remains to show the density of 
the periodic points. Lets denote by $P_1$ (resp. $P_2$) the set of periodic points
of the operator $T(p_1/q_1)$ (resp. $S(p_2/q_2)$). These are dense linear spaces and hence, $P_1\otimes P_2$ is dense in $X\tilde{\otimes}_{\alpha}Y$. Furthermore, if 
$x\otimes y \in P_1\otimes P_2$ there exist 
$n,m\in \N$ with  $T(p_1/q_1)^nx=T(np_1/q_1)x=x$, $S(p_2/q_2)^m y=S(mp_2/q_2)y=y$,
and $np_1/q_1 = mp_2/q_2 =:t$.
Then we have
$$
 T(t)\otimes S(t)(x\otimes y) =  T(t)x\otimes S(t)y = x\otimes y
$$
and $x\otimes y$ is therefore a periodic point of $T(t)\otimes S(t)$. Since
$$
P_1\otimes P_2 = \spa\left\{ x_k\otimes y_k : x_k\in P_1, y_k \in P_2\right\}
$$ 
and as with
$x_k\otimes y_k, k=1,\ldots, n$, also $\sum_{k=1}^n x_k\otimes y_k$ is a periodic point,  
$P_1\otimes P_2$ consists only of periodic points.
\end{proof}

\subsubsection*{Acknowledgement}
I am indebted to the reviewer who made many valuable comments that greatly improved the
exposition of this paper.

\bibliographystyle{amsplain}
%\addcontentsline{toc}{chapter}{Bibliography}
\bibliography{dissertation,hypercyclic}

\end{document}